\documentclass[12pt]{amsart}
\newtheorem{theorem}{Theorem}[section]
\newtheorem{lemma}[theorem]{Lemma}

\newtheorem{proposition}[theorem]{Proposition}
\theoremstyle{definition}
\newtheorem{definition}[theorem]{Definition}
\newtheorem{example}[theorem]{Example}

\usepackage{enumerate}

\usepackage[margin=0.75in]{geometry}
\usepackage{calligra}
\usepackage{xcolor, soul}
\sethlcolor{yellow}
\usepackage{mathtools}

\DeclareMathOperator{\sinc}{sinc}

\DeclareMathOperator\supp{supp}
\DeclareMathOperator\Span{span}
\DeclareMathOperator\tr{tr}
\DeclareMathOperator\re{Re}

\numberwithin{equation}{section}

\def\H{{\mathbb{H}}}
\def\L{{L^2(\mathbb{H})}}
\def\N{{\mathbb{N}}}
\def\h{{\mathcal{H}}}
\def\R{{\mathbb{R}}}

\def\b{{\mathcal{B}_2}}

\def\z{{\mathbb{Z}}}

\def\f{{\mathcal{F}}}

\def\p{{\phi}}

\newcommand\numberthis{\addtocounter{equation}{1}\tag{\theequation}}

\newcommand{\be}{\begin{equation}}
\newcommand{\ee}{\end{equation}}
\begin{document}
\title{System of left translates and Oblique dual on the Heisenberg group}

%    Information for first author
\author{S. R. Das}
%    Address of record for the research reported here
\address{School of Mathematical Sciences, NISER Bhubaneswar, Jatni, Odisha 752050, India}
%    Current address
%\curraddr{Department of Mathematics and Statistics,
%Case Western Reserve University, Cleveland, Ohio 43403}
\email{santiranjandas100@gmail.com}
%    \thanks will become a 1st page footnote.
%\thanks{The first author was supported in part by NSF Grant \#000000.}
\author{P. Massopust}
\address{Center of Mathematics, Technical University of Munich, Germany}
\email{massopust@ma.tum.de}

%    Information for second author
\author{R. Radha}
\address{Department of Mathematics, Indian Institute of Technology, Madras, India}
\email{radharam@iitm.ac.in}
%\thanks{Support information for the second author.}
%    General info

%\date{January 1, 2001 and, in revised form, June 22, 2001.}

%\dedicatory{This paper is dedicated to our advisors.}

\begin{abstract}
In this paper, we characterize the system of left translates $\{L_{(2k,l,m)}g:k,l,m\in\z\}$, $g\in L^2(\H)$, to be a frame sequence or a \emph{Riesz} sequence in terms of the twisted translates of the corresponding function $g^\lambda$. Here, $\H$ denotes the Heisenberg group and $g^\lambda$ the inverse Fourier transform of $g$ with respect to the central variable. This type of characterization for a \emph{Riesz} sequence allows us to find some concrete examples. We also study the structure of the oblique dual of the system of left translates $\{L_{(2k,l,m)}g:k,l,m\in\z\}$ on $\H$. This result is also illustrated with an example.
%we introduce a class of $B$-splines on the \emph{Heisenberg} group $\H$ and study their fundamental properties. Unlike the classical case, we prove that there does not exist any sequence $\{\alpha_n\}_{n\in\N}$ such that $L_{(-n.-\frac{n}{2},-\alpha_n)}\p_n(x,y,t)=L_{(-n.-\frac{n}{2},-\alpha_n)}\p_n(-x,-y,-t)$, for $n\geq 2$, where $L_{(x,y,t)}$ denotes the left translation on $\H$. We further investigate the problem of finding an equivalent condition for the system of left translates to form a frame sequence or a \emph{Riesz} sequence in terms of twisted translates. We also find a sufficient condition for obtaining an oblique dual of the system $\{L_{(2k,l,m)}g:k,l,m\in\z\}$ for a certain class of functions $g\in\L$. These concepts are illustrated by some examples. Finally, we make some remarks about $B$-splines regarding these results.
\end{abstract}
\maketitle
\section{Introduction}

A closed subspace $V\subset L^2(\R)$ is said to be a shift-invariant space if $f\in V\Rightarrow \mathcal{T}_kf\in V$ for any $k\in\z$, where $\mathcal{T}_xf(y)=f(y-x)$ denotes the translation operator. These spaces appear in the study of multiresolution analyses in order to construct wavelets. We refer to \cite{mallat, meyer} in this context. For $\phi\in L^2(\R)$, the shift-invariant space $V(\phi)=\overline{\Span\{\mathcal{T}_k\phi:k\in\z}\}$ is called a principal shift-invariant space. Shift-invariant spaces are broadly applied in various fields such as approximation theory, mathematical sampling theory, communication engineering, and so on. Apart from this, shift-invariant spaces have also been explored in various group settings.\\

In \cite{bownik}, Bownik obtained a characterization of shift-invariant spaces on $\R^n$ by using range functions. He derived equivalent conditions for a system of translates to be a frame sequence or a \emph{Riesz} sequence. Later these results were studied on locally compact abelian groups in \cite{bownikr, cabrelli, jakobsen,kamyabi} and on non-abelian compact groups in \cite{iverson, radhas}.\\

In recent years, problems in connection with frames, \emph{Riesz} bases, wavelets, and shift-invariant spaces on non-abelian groups, nilpotent \emph{Lie} groups, especially the \emph{Heisenberg} group, have drawn the attention of several researchers globally. (See, for example, \cite{aratiindag, aratijmpa, araticol, bhm, cmo, saswatahouston} in this context.)\\

In \cite{santifirst}, Das et al. obtained characterization results for a shift-invariant system to be a frame sequence or a \emph{Riesz} sequence in terms of the Gramian and the dual Gramian, respectively, on the Heisenberg group. Although the characterization results mentioned in this paper are interesting from the theoretical point of view, they are not useful in obtaining concrete \emph{Riesz} sequences of system of translates. In this paper, we attempt to overcome this difficulty and try to obtain a characterization for the system of left translates on the Heisenberg group to form a frame sequence or a \emph{Riesz} sequence. This is done with the help of deriving such characterizations for $\lambda$-twisted translates on $\R^2$. Apart from this, we also study the problem of obtaining oblique dual for a system of left translates on the Heisenberg group.\\

The structure of this paper is as follows. After introducing some background information about frames and the Heisenberg group in Section 2, we consider systems of left translates and their relation to frame and Riesz sequences on the \emph{Heisenberg} group in Section 3. Obliques duals of these systems of left translates are then investigated in Section 4.

\section{Background}
To proceed, we require the following definitions and results from frame theory and harmonic analysis on the Heisenberg group. In the former case, most of these can be found in, for instance, \cite{C}, and in the latter case in, i.e., \cite{follandphase,thangavelu}. 

$0\neq\h$ always denotes a separable Hilbert space.
\begin{definition}
A sequence $\{f_k:~k\in\N\}\subset \h$ is said to be a frame for $\h$ if there exist constants $A,B>0$ satisfying
\begin{align}\label{fc}
A\|f\|^2\leq\sum_{k\in\N}\mid\langle f,f_k\rangle\mid^2\hspace{1 mm}\leq B\|f\|^2,\hspace{.75 cm}\forall\ f\in\h.
\end{align}
\end{definition}
If $\{f_k:~k\in\N\}$ is a frame for $\overline{\Span\{f_k:~k\in\N\}}$, then it is called a frame sequence. 

A sequence $\{f_k:~k\in\N\}\subset \h$ satisfying only the upper bound in the frame condition \eqref{fc} is called a \emph{Bessel} sequence.

\begin{definition}
A sequence of the form $\{Ue_k:~k\in\N\}$, where $\{e_k:~k\in\N\}$ is an orthonormal basis of $\h$ and $U$ is a bounded invertible operator on $\h$, is called a \emph{Riesz} basis. If $\{f_k:~k\in\N\}$ is a \emph{Riesz} basis for $\overline{\Span\{f_k:~k\in\N\}}$, then it is called a \emph{Riesz} sequence. 
\end{definition}
Equivalently, $\{f_k:~k\in\N\}$ is said to be a \emph{Riesz} sequence if there exist constants $A,B>0$ such that
\begin{align*}
A\|\{c_{k}\}\|_{\ell^2(\N)}^2\leq\bigg\|\sum_{k\in\N}c_kf_k\bigg\|^2\leq B\|\{c_{k}\}\|_{\ell^2(\N)}^2,
\end{align*}
for all finite sequences $\{c_{k}\}\in \ell^2(\N)$.
\begin{theorem}\label{h}
Let $h\in L^2(\R)$. The system $\{T_kh:k\in\z\}$ is a \emph{Riesz} sequence with bounds $A,B>0$ iff
\begin{align*}
A\leq\sum_{k\in\z}|\widehat{h}(\lambda+k)|^2\leq B,\ \ \ \ for\ a.e.\ \lambda\in(0,1].
\end{align*}
\end{theorem}
\begin{definition}
The {\it Gramian} $G$ associated with a {\it Bessel} sequence $\{f_k:~k\in\N\}$ is a bounded operator on $\ell^2(\N)$ defined by
\begin{align*}
G\{c_{k}\} :=\bigg\{\sum_{k\in\N}\big\langle f_k,f_j\big\rangle c_k\bigg\}_{j\in\N}.
\end{align*}
\end{definition}
It is well known that $\{f_k:~k\in\N\}$ is a {\it Riesz} sequence with bounds $A,B > 0$ iff
\begin{align*}
A\|\{c_{k}\}\|_{\ell^2(\N)}^2\leq\big\langle G\{c_{k}\},\{c_{k}\}\big\rangle\leq B\|\{c_{k}\}\|_{\ell^2(\N)}^2.
\end{align*}
\begin{definition}
Let $\{f_k:k\in\N\}$ be a \emph{Riesz} sequence in $\h$. If 
$$f=\sum\limits_{k\in\N}\langle f,g_k\rangle f_k,\ \ \ \ \forall\ f\in \overline{\Span\{f_k:k\in\z\}},$$
for some $\{g_k:k\in\N\} \subset \h$, then $\{g_k:k\in\N\}$ is called a generalized dual generator of $\{f_k:k\in\N\}$. In addition, if $\{g_k:k\in\N\}$ is a frame sequence, then $\{g_k:k\in\N\}$ is called an oblique dual generator of $\{f_k:k\in\N\}$.  
\end{definition}
\begin{definition}
Let $\{f_k:k\in J\}$ be a countable collection of elements in $\h$ and $\{\alpha_k\}_{k\in J}\in \ell^2(J)$. Consider the system of equations
\begin{equation}\label{4}
\langle f,f_k\rangle=\alpha_k,\ \ \ \forall\ k\in J.
\end{equation}
Finding such an $f\in\h$ from \eqref{4} is known as the moment problem.
\end{definition}
A moment problem may not have any solution at all or may have infinitely many solutions. But if $\{f_k:k\in J\}$ is a \emph{Riesz} sequence, then the moment problem has a unique solution $f\in\overline{\Span\{f_k:k\in J\}}$. For the existence of a solution of a moment problem one has the following result. 
\begin{lemma}[\cite{okl}]\label{c}
Let $\{f_k:k=1,2,\cdots,N\}$ be a finite collection of vectors in $\h$. Consider the moment problem
\begin{align*}
\langle f,f_k\rangle=\delta_{k,1}\ ,\ \ \ k=1,2,\cdots,N.
\end{align*}
Then the following statements are equivalent:
\begin{enumerate}
\item[\emph{(i)}] The moment problem has a solution $f\in\h$.
\item[\emph{(ii)}] $\sum\limits_{k=1}^N c_kf_k=0$, for some $\{c_k\}$ implies $c_1=0$.
\item[\emph{(iii)}] $f_1\notin \Span\{f_2, f_3,\cdots, f_N\}$.
\end{enumerate}
\end{lemma}  
\begin{definition}
A closed subspace $V\subset L^2(\R)$ is called a shift-invariant space if $f\in V\Rightarrow T_kf\in V$ for any $k\in\z$, where $T_x$ denotes the translation operator $T_xf(y)=f(y-x)$. In particular, if $\p\in L^2(\R)$, then $V(\p)=\overline{\Span\ \{T_k\p:k\in\z\}}$ is called a principal shift-invariant space.
\end{definition}
For a study of frames, {\it Riesz} basis on $\mathcal{H}$, and shift-invariant spaces on $L^2(\R)$, we refer to \cite{C}.\\
\begin{definition}
Let $\chi$ denote the characteristic function of $[0,1]$. For $n\in \N$, set
\begin{align*}
&B_{1}:[0,1]\to [0,1], \quad x\mapsto \chi(x);\\
&B_{n} := B_{n-1}\ast B_{1}, \quad n\geq 2,\quad n\in \N.\numberthis \label{1}
\end{align*}
Then $B_n$ is called a \emph{(cardinal) polynomial B-spline of order $n$}.
\end{definition}
For more and detailed information about B-splines and their applications, the interested reader may wish to consult any of the many references regarding B-splines.

The next stated result shows that cardinal B-splines form principal shift-invariant spaces.

\begin{theorem}\cite[Theorem 9.2.6]{C}\label{f}
For each $n\in \N$, the sequence $\{T_kB_n\}_{k\in \z}$ is a Riesz sequence.
\end{theorem}
%The {\it Heisenberg} group $\H$ is a nilpotent {\it Lie} group whose underlying manifold is $\R\times\R\times\R$ endowed with a group operation defined by 
%\begin{align*}
%(x,y,t)(u,v,s)=(x+u,u+v,t+s+\frac{1}{2}(uy-vx)).
%\end{align*}
%The {\it Haar} measure on $\H$ is the {\it Lebesgue} measure $dxdydt$ on $\R^3$. The classical Stone-von Neumann theorem states that every infinite dimensional irreducible unitary representation of $\H$ is unitarily equivalent to the {\it Schr\"{o}dinger} representation $\pi_\lambda$, $\lambda\in\R^*$, given by
%\begin{equation*}
%\pi_\lambda (x,y,t)\phi (\xi)=e^{2\pi i\lambda t}e^{2\pi i\lambda (x\xi +\frac{1}{2} xy)}\phi(\xi +y),~~\phi\in L^2(\R).
%\end{equation*}
The {\it Heisenberg} group $\mathbb{H}$ is a nilpotent {\it Lie} group whose underlying manifold is $\R\times\R\times\R$ endowed with a group operation defined by 
\[
(x,y,t)(x^\prime,y^\prime,t^\prime):=(x+x^\prime,y+y^\prime,t+t^\prime+\tfrac{1}{2}(x^\prime y-y^\prime x)),
\]
and where {\it Haar} measure is {\it Lebesgue} measure $dx\,dy\,dt$ on $\R^3$. By the {\it Stone--von Neumann} theorem, every infinite dimensional irreducible unitary representation on $\mathbb{H}$ is unitarily equivalent to the representation $\pi_\lambda$ given by
\begin{align*}
\pi_\lambda(x,y,t)\p(\xi)=e^{2\pi i\lambda t}e^{2\pi i\lambda(x\xi+\frac{1}{2}xy)}\p(\xi+y),
\end{align*}
for $\p\in L^2(\R)$ and $\lambda\in \R^\times:=\R\setminus\{0\}$. This representation $\pi_\lambda$ is called the {\it Schr\"{o}dinger} representation of the {\it Heisenberg} group. For $f,g\in L^1(\mathbb{H})$, the group convolution of $f$ and $g$ is defined by
\begin{align}\label{hconv}
f*g(x,y,t):=\int_{\mathbb{H}}f\big((x,y,t)(u,v,s)^{-1}\big)g(u,v,s)\ du\,dv\,ds.
\end{align}
Under this group convolution, $L^1(\H)$ becomes a non-commutative \emph{Banach} algebra.\\
The group $Fourier$ transform of $f\in L^1(\H)$ is defined by
\begin{equation}\label{2}
\widehat{f}(\lambda)=\int_\H f(x,y,t)\hspace{1 mm}\pi_\lambda (x,y,t)\hspace{1 mm}dxdydt,\hspace{3 mm}\lambda\in\R^\times,
\end{equation}
where the integral is a $Bochner$ integral acting on the $Hilbert$ space $L^2(\R)$. The group {\it Fourier} transform is an isometric isomorphism between $\L$ and $L^2(\R^\times,\b;d\mu)$, where  $d\mu (\lambda)$ denotes \emph{Plancherel} measure $|\lambda |d\lambda$ and $\b$ is the {\it Hilbert} space of {\it Hilbert-Schmidt} operators on $L^2(\R)$ with inner product given by $(T,S) :=\tr(TS^*)$. Thus, we can write \eqref{2} as
\begin{equation*}
\widehat{f}(\lambda)=\int_{\R^2}f^\lambda(x,y)\pi_\lambda(x,y,0)\ dxdy\ ,
\end{equation*}
where
\begin{equation*}
f^\lambda(x,y) :=\int_{\R}f(x,y,t)e^{2\pi i\lambda t}\ dt.
\end{equation*} 
Note that the function $f^\lambda(x,y)$ is the inverse {\it Fourier} transform of $f$ with respect to the $t$ variable. For $g\in L^1(\R^2)$, let
\begin{equation*}
W_{\lambda}(g):=\int_{\R^2}g(x,y)\pi_{\lambda}(x,y,0)\ dxdy,\ \ \text{for $\lambda\in\R^\times$.}
\end{equation*}
Using this operator, we can rewrite $\widehat{f}(\lambda)$ as $W_{\lambda}(f^{\lambda})$. When $f,g\in\L$, one can show that $f^{\lambda}, g^{\lambda}\in L^2(\R^2)$ and $W_{\lambda}$ satisfies
\begin{equation}\label{3}
\big\langle f^\lambda,g^\lambda\big\rangle_{L^2(\R^2)}=|\lambda|\big\langle W_\lambda (f^\lambda),W_\lambda (g^\lambda)\big\rangle_{\b}.
\end{equation}
%In particular, when $\lambda=1$, $W_{\lambda}$ is called the {\it Weyl} transform. Moreover, $\widehat{f}(\lambda)$ is an integral operator whose kernel $K_{f^\lambda}^\lambda$ is given by
%\begin{align*}
%K_{f^\lambda}^\lambda(\xi,\eta)=\int_\R f^\lambda(x,\eta-\xi)e^{\pi i\lambda x(\xi+\eta)}~dx.
%\end{align*}
Now, define $\tau:\L\rightarrow L^2((0,1],\ell^2(\z,\b))$ by 
\[
\tau f(\lambda):=\{|\lambda-r|^{1/2}\widehat{f}(\lambda-r)\}_{r\in\z}, \quad\forall\ f\in\L,\; \lambda\in(0,1]. 
\]
Then, $\tau$ is an isometric isomorphism between $\L$ and $L^2((0,1],\ell^2(\z,\b))$. (See \cite{cmo, santifirst} in this context.) For $(u,v,s)\in\H$, the left translation operator $L_{(u,v,s)}$ is defined by 
\[
L_{(u,v,s)}f(x,y,t) :=f((u,v,s)^{-1}(x,y,t)),\hspace{1mm}\quad\forall\, (x,y,t)\in\H,
\]
which is a unitary operator on $\L$. Using the definitions of the left translation operator and the convolution, one can show that
\begin{equation}\label{28}
L_{(u,v,s)}(f*g)=(L_{(u,v,s)}f)*g.
\end{equation}
For $(u,v)\in\R^2$ and $\lambda\in\R^\times$, the $\lambda$-twisted translation operator $(T_{(u,v)}^t)^\lambda$ is defined by
\begin{align*}
(T_{(u,v)}^t)^\lambda F(x,y) :=e^{\pi i\lambda(vx-uy)}F(x-u,y-v),\quad \forall\ (x,y)\in\R^2,
\end{align*}
which is also a unitary operator on $L^2(\R^2)$. It is easy to see that
\begin{align}\label{23}
(L_{(u,v,s)}f)^\lambda=e^{2\pi is\lambda}(T_{(u,v)}^t)^\lambda f^\lambda.
\end{align}
For further properties of $\lambda$-twisted translation, we refer to \cite{saswatahouston}. 

Recall that for a locally compact group $G$, a lattice $\Gamma$ in $G$ is defined to be a discrete subgroup of $G$ which is co-compact. The standard lattice in $\H$ is taken to be $\Gamma:=\{(2k,l,m):k,l,m\in\z\}$. For a study of analysis on the \emph{Heisenberg} group we refer to \cite{follandphase, thangavelu}.
\section{System of left translates as a frame sequence and a \emph{Riesz} sequence}
Let $g\in\L$. In this section, we wish to obtain characterization results for the system $\{L_{(2k,l,m)}g:k,l,m\in\z\}$ to form a frame sequence or a \emph{Riesz} sequence in terms of the $\lambda$-twisted translations $g^\lambda$ of $g$.\\ 

From Corollary $3$ of \cite{santifirst}, we know that $\{L_{(2k,l,m)}g:k,l,m\in\z\}$ is a frame sequence with bounds $A,B>0$ iff
\begin{align}\label{22}
A\|\Phi(\lambda)\|^2\leq \sum_{k,l\in\z}|\langle \Phi(\lambda),\tau(L_{(2k,l,0)}g)(\lambda)\rangle|^2\leq B\|\Phi(\lambda)\|^2,\ \forall\ \Phi(\lambda)\in J(\lambda),\ \text{for a.e. $\lambda\in(0,1]$},
\end{align}
where $J(\lambda):=\overline{\Span\{\tau(L_{(2k,l,0)}g)(\lambda):k,l\in\z\}}$.

In order to prove that $\{L_{(2k,l,m)}g:k,l,m\in\z\}$ is a frame sequence, it suffices to consider the class $\Span\{\tau(L_{(2k,l,0)}g)(\lambda):k,l\in\z\}$ instead of $J(\lambda)$. Thus, the required condition for the verification of frame sequence reduces to the following two inequalities:

For any finite $\f\subset\z^2$ and any finite sequence $\{\alpha_{k,l}\}\in\ell^2(\z^2)$,
\begin{align*}
A\bigg\|\sum_{(k,l)\in\f}\alpha_{k,l}\tau(L_{(2k,l,0)}g)(\lambda)\bigg\|_{\ell^2(\z,\b)}^2 & \leq \sum_{k,l\in\z}\bigg|\Big\langle\sum_{(k^\prime,l^\prime)\in\f}\alpha_{k^\prime,l^\prime}\tau(L_{(2k^\prime,l^\prime,0)}g)(\lambda),\tau(L_{(2k,l,0)}g)(\lambda)\Big\rangle\bigg|^2\\
& \leq B\bigg\|\sum_{(k,l)\in\f}\alpha_{k,l}\tau(L_{(2k,l,0)}g)(\lambda)\bigg\|_{\ell^2(\z,\b)}^2,\ \text{a.e.}\ \lambda\in(0,1].
\end{align*}
Now, for $k,k^\prime,l,l^\prime\in\z$,
\begin{align*}
\langle\tau(L_{(2k^\prime,l^\prime,0)}g)(\lambda),\tau(L_{(2k,l,0)}g)(\lambda)\rangle_{\ell^2(\z,\b)}&=\sum_{r\in\z}|\lambda-r|\ \big\langle\widehat{L_{(2k^\prime,l^\prime,0)}g}(\lambda-r),\widehat{L_{(2k,l,0)}g}(\lambda-r)\big\rangle_\b\\
=\sum_{r\in\z}&|\lambda-r|\ \big\langle W_{\lambda-r}\big((L_{(2k^\prime,l^\prime,0)}g)^{\lambda-r}\big),W_{\lambda-r}\big((L_{(2k,l,0)}g)^{\lambda-r}\big)\big\rangle_\b\\
=\sum_{r\in\z}&\big\langle(L_{(2k^\prime,l^\prime,0)}g)^{\lambda-r},(L_{(2k,l,0)}g)^{\lambda-r}\big\rangle_{L^2(\R^2)}\\
=\sum_{r\in\z}&\big\langle(T_{(2k^\prime,l^\prime)}^t)^{\lambda-r}g^{\lambda-r},(T_{(2k,l)}^t)^{\lambda-r}g^{\lambda-r}\big\rangle_{L^2(\R^2)}\\
=\sum_{r\in\z}&e^{\pi i(\lambda-r)(kl^\prime-lk^\prime)}\big\langle\big(T_{(2(k^\prime-k),l^\prime-l)}^t\big)^{\lambda-r}g^{\lambda-r},g^{\lambda-r}\big\rangle_{L^2(\R^2)},\numberthis \label{26n}
\end{align*}
where we used \eqref{3} and \eqref{23}. 

In the following theorem, we state a condition for the system $\{L_{(2k,l,m)}g:k,l,m\in\z\}$ to be a frame sequence in terms of the $\lambda$-twisted translates of $g^\lambda$ on $\R^2$.
\begin{theorem}\label{a}
Let $g\in\L$. Then, the system $\{L_{(2k,l,m)}g:k,l,m\in\z\}$ is a frame sequence with bounds $A,B>0$ iff
\begin{align*}
A\sum_{r\in\z}\bigg\|\sum_{(k^\prime,l^\prime)\in\f}\alpha_{k^\prime,l^\prime}&(T_{(2k^\prime,l^\prime)}^t)^{\lambda-r}g^{\lambda-r}\bigg\|_{L^2(\R^2)}^2\leq \sum_{k,l\in\z}\bigg|\sum_{(k^\prime,l^\prime)\in\f,r\in\z}\alpha_{k^\prime,l^\prime}e^{\pi i(\lambda-r)(kl^\prime-lk^\prime)}\times\\
\Big\langle (T_{(2(k^\prime-k),l^\prime-l)}^t)^{\lambda-r}&g^{\lambda-r},g^{\lambda-r}\Big\rangle_{L^2(\R^2)}\bigg|^2\leq B\sum_{r\in\z}\bigg\|\sum_{(k^\prime,l^\prime)\in\f}\alpha_{k^\prime,l^\prime}(T_{(2k^\prime,l^\prime)}^t)^{\lambda-r}g^{\lambda-r}\bigg\|_{L^2(\R^2)}^2,\ \ \text{a.e.}\ \lambda\in(0,1],
\end{align*}
for any finite $\f\subset\z^2$ and any finite sequence $\{\alpha_{k,l}\}\in\ell^2(\z^2)$.
\end{theorem}
\begin{proof}
The system $\{L_{(2k,l,m)}g:k,l,m\in\z\}$ is a frame sequence with bounds $A,B>0$ iff \eqref{22} holds. Consider $\Phi(\lambda) :=\sum\limits_{(k,l)\in\f}\alpha_{k,l}\tau(L_{(2k,l,0)}g)(\lambda)$, for some finite $\f\subset\z^2$ and a finite sequence $\{\alpha_{k,l}\}\in\ell^2(\z^2)$. Then,
\begin{align*}
\|\Phi(\lambda)\|_{\ell^2(\z,\b)}^2&=\bigg\|\sum_{(k,l)\in\f}\alpha_{k,l}\tau(L_{(2k,l,0)}g)(\lambda)\bigg\|_{\ell^2(\z,\b)}^2\\
&=\bigg\|\Big\{|\lambda-r|^{1/2}\sum_{(k,l)\in\f}\alpha_{k,l}\widehat{L_{(2k,l,0)}g}(\lambda-r)\Big\}_{r\in\z}\bigg\|_{\ell^2(\z,\b)}^2\\
&=\sum_{r\in\z}|\lambda-r|\ \bigg\|\sum_{(k,l)\in\f}\alpha_{k,l}\widehat{L_{(2k,l,0)}g}(\lambda-r)\bigg\|_\b^2\\
&=\sum_{r\in\z}|\lambda-r|\ \bigg\|\sum_{(k,l)\in\f}\alpha_{k,l}W_{\lambda-r}\big((L_{(2k,l,0)}g)^{\lambda-r}\big)\bigg\|_\b^2.
\end{align*}
Employing \eqref{3} and \eqref{23}, yields
\begin{align*}
\|\Phi(\lambda)\|_{\ell^2(\z,\b)}^2&=\sum_{r\in\z}\bigg\|\sum_{(k,l)\in\f}\alpha_{k,l}(L_{(2k,l,0)}g)^{\lambda-r}\bigg\|_{L^2(\R^2)}^2\\
&=\sum_{r\in\z}\bigg\|\sum_{(k,l)\in\f}\alpha_{k,l}(T_{(2k,l)}^t)^{\lambda-r}g^{\lambda-r}\bigg\|_{L^2(\R^2)}^2.\numberthis \label{24}
\end{align*}
On the other hand,
\begin{align}\label{25}
\langle\Phi(\lambda),\tau(L_{(2k,l,0)}g)(\lambda)\rangle=\sum_{(k^\prime,l^\prime)\in\f}\alpha_{k^\prime,l^\prime}\langle\tau(L_{(2k^\prime,l^\prime,0)}g)(\lambda),\tau(L_{(2k,l,0)}g)(\lambda)\rangle,
\end{align}
for $k,l\in\z$.
%Thus,
%\begin{align}\label{26}
%\langle\tau(L_{(2k^\prime,l^\prime,0)}g)(\lambda),\tau(L_{(2k,l,0)}g)(\lambda)\rangle_{\ell^2(\z,\b)}=\sum_{r\in\z}e^{\pi i(\lambda-r)(kl^\prime-lk^\prime)}\big\langle\big(T_{(2(k^\prime-k),l^\prime-l)}^t\big)^{\lambda-r}g^{\lambda-r},g^{\lambda-r}\big\rangle_{L^2(\R^2)}.
%\end{align}
Using \eqref{26n} in \eqref{25}, we obtain
\begin{align}\label{27}
\langle\Phi(\lambda),\tau(L_{(2k,l,0)}g)(\lambda)\rangle=\sum_{(k^\prime,l^\prime)\in\f,r\in\z}\alpha_{k^\prime,l^\prime}e^{\pi i(\lambda-r)(kl^\prime-lk^\prime)}\big\langle\big(T_{(2(k^\prime-k),l^\prime-l)}^t\big)^{\lambda-r}g^{\lambda-r},g^{\lambda-r}\big\rangle_{L^2(\R^2)}.
\end{align}
Employing \eqref{24} and \eqref{27} in \eqref{22}, the required result follows.  
\end{proof} 
Next, we aim to characterize the system of left translates $\{L_{(2k,l,m)}g:k,l,m\in\z\}$ to be a \emph{Riesz} sequence, again in terms of $\lambda$-twisted translates of $g^\lambda$. To this end, we consider the Gramian associated with the system $\{\tau(L_{(2k,l,0)}g)(\lambda):k,l\in\z\}$ and obtain an equivalent condition for a {\it Riesz} sequence.\\

First, consider the Gramian associated with the system $\{\tau(L_{(2k,l,0)}g)(\lambda):k,l\in\z\}$. For $g\in\L$ and $\lambda\in(0,1]$, the Gramian of $\{\tau(L_{(2k,l,0)}g)(\lambda):k,l\in\z\}$ is defined by \[
G(\lambda):=H(\lambda)^*H(\lambda):\ell^2(\z^2)\rightarrow\ell^2(\z^2), 
\]
where $H(\lambda):\ell^2(\z^2)\rightarrow \ell^2(\z,\b)$ is given by 
\[
H(\lambda)\big(\{c_{k,l}\}\big):=\sum\limits_{k,l\in\z}c_{k,l}\tau(L_{(2k,l,0)}g)(\lambda). 
\] 
%Using \eqref{26}, we get
%\begin{align*}
%\big\langle\tau(L_{(2k,l,0)}g)(\lambda),\tau(L_{(2k^\prime,l^\prime,0)}g)(\lambda)\big\rangle_{\ell^2(\z,\b)}&=\sum_{r\in\z}e^{2\pi i(\lambda-r)(lk^\prime-kl^\prime)}\big\langle\big(T_{(2(k-k^\prime),l-l^\prime)}^t\big)^{\lambda-r}g^{\lambda-r},g^{\lambda-r}\big\rangle_{L^2(\R^2)}.
%\end{align*}
We obtain the following
\begin{theorem}\label{d}
The system $\{L_{(2k,l,m)}g:k,l,m\in\z\}$ is a \emph{Riesz} sequence iff there exists $A,B>0$ such that
\begin{align*}
A\|\{c_{k,l}\}\|_{\ell^2(\z^2)}^2 &\leq\sum_{k,l,k^\prime,l^\prime\in\z}\sum_{r\in\z}c_{k,l}\overline{c}_{k^\prime,l^\prime}e^{2\pi i(\lambda-r)(lk^\prime-kl^\prime)}\big\langle\big(T_{(2(k-k^\prime),l-l^\prime)}^t\big)^{\lambda-r}g^{\lambda-r},g^{\lambda-r}\big\rangle_{L^2(\R^2)}\\
&\leq B\|\{c_{k,l}\}\|_{\ell^2(\z^2)}^2,\numberthis \label{9}
\end{align*}
for $a.e.\ \lambda\in(0,1]$ and for all $\{c_{k,l}\}\in\ell^2(\z^2)$.
\end{theorem}
\begin{proof}
By Theorem $6$ of \cite{santifirst}, the system $\{L_{(2k,l,m)}g:k,l,m\in\z\}$ is a \emph{Riesz} sequence iff there exist $A,B>0$ such that
\begin{align*}
A\|\{c_{k,l}\}\|_{\ell^2(\z^2)}^2\leq\langle G(\lambda)\{c_{k,l}\},\{c_{k,l}\}\rangle_{\ell^2(\z^2)}\leq B\|\{c_{k,l}\}\|_{\ell^2(\z^2)}^2, \numberthis \label{1201n}
\end{align*}
for a.e. $\lambda\in(0,1]$ and for all $\{c_{k,l}\}\in\ell^2(\z^2)$. But
\begin{align*}
\langle G(\lambda)\{c_{k,l}\},\{c_{k,l}\}\rangle_{\ell^2(\z^2)}&=\big\|H(\lambda)\big(\{c_{k,l}\}\big)\big\|_{\ell^2(\z,\b)}^2\\
&=\bigg\|\sum_{k,l\in\z}c_{k,l}\tau(L_{(2k,l,0)}g)(\lambda)\bigg\|_{\ell^2(\z,\b)}^2\\
&=\sum_{k,l,k^\prime,l^\prime\in\z}c_{k,l}\bar{c}_{k^\prime,l^\prime}\big\langle\tau(L_{(2k,l,0)}g)(\lambda),\tau(L_{(2k^\prime,l^\prime,0)}g)(\lambda)\big\rangle_{\ell^2(\z,\b)},\\
&=\sum_{k,l,k^\prime,l^\prime\in\z}\sum_{r\in\z}c_{k,l}\overline{c}_{k^\prime,l^\prime}e^{2\pi i(\lambda-r)(lk^\prime-kl^\prime)}\big\langle\big(T_{(2(k-k^\prime),l-l^\prime)}^t\big)^{\lambda-r}g^{\lambda-r},g^{\lambda-r}\big\rangle_{L^2(\R^2)},\numberthis \label{1200n}
\end{align*}
by using \eqref{26n}. Employing \eqref{1200n} in \eqref{1201n}, we obtain \eqref{9}.
\end{proof}
\begin{example}
Let $\p(x,y,t):=\chi_{[0,2]}(x)\chi_{[0,2]}(y)h(t)$, where $\chi_{[0,2]}$ denotes the characteristic function on $[0,2]$ and $h\in L^2(\R)$ is given by $\widehat{h}(\lambda)=\chi_{[0,p]}(\lambda)$, for $\N\ni p\geq 3$. Then, $\left\|\p\right\|_{\L}^2=4\left\|h\right\|_{L^2(\R)}^2$. Furthermore, $\p^\lambda(x,y)=\chi_{[0,2]}(x)\chi_{[0,2]}(y)\widehat{h}(-\lambda)$. Now,
\begin{align*}
\left\langle\left(T_{(2k,l)}^t\right)^\lambda\p^\lambda,\p^\lambda\right\rangle&=\int_{\R^2}e^{\pi i\lambda(lx-2ky)}\p^\lambda(x-2k,y-l)\overline{\p^\lambda(x,y)}\,dx\,dy\\
&=\overline{\widehat{h}(-\lambda)}\int_0^2\int_0^2e^{\pi i\lambda(lx-2ky)}\p^\lambda(x-2k,y-l)\,dy\,dx\\
&=\overline{\widehat{h}(-\lambda)}\int_{-2k}^{2-2k}\int_{-l}^{2-l}e^{\pi i\lambda(lx-2ky)}\p^\lambda(x,y)\,dy\,dx\\
&=|\widehat{h}(-\lambda)|^2\int_{[-2k,2-2k]\cap[0,2]}\int_{[-l,2-l]\cap[0,2]}e^{\pi i\lambda(lx-2ky)}\,dy\,dx.\numberthis \label{37}
\end{align*}
For $\lambda\in(0,1]$ and $\{c_{k,l}\}\in\ell^2(\z^2)$, consider the middle term in \eqref{9} which is $\langle G(\lambda)\{c_{k,l}\},\{c_{k,l}\}\rangle_{\ell^2(\z^2)}$.
%\begin{align*}
%\langle G(\lambda)\{c_{k,l}\},\{c_{k,l}\}\rangle_{\ell^2(\z^2)}=\sum_{k,l,k^\prime,l^\prime\in\z}\sum_{r\in\z}c_{k,l}\overline{c}_{k^\prime,l^\prime}e^{2\pi i(\lambda-r)(lk^\prime-kl^\prime)}\big\langle\big(T_{(2(k-k^\prime),l-l^\prime)}^t\big)^{\lambda-r}&\p^{\lambda-r},\p^{\lambda-r}\big\rangle_{L^2(\R^2)}.\numberthis \label{38}
%\end{align*}
It follows from \eqref{37} that only $k^\prime=k$ and $l^\prime=l-1,\,l,\,l+1$ will contribute to the sum over $k^\prime,l^\prime\in\z$. Thus, we have
\begin{align*}
\langle G(\lambda)\{c_{k,l}\},\{c_{k,l}\}\rangle_{\ell^2(\z^2)}=M_1+M_2+M_3,
\end{align*}
where
\begin{align*}
M_1:=\sum_{r\in\z}\sum_{k,l\in\z}c_{k,l}\overline{c_{k,l-1}}e^{2\pi i(\lambda-r)k}\left\langle\left(T_{(0,1)}^t\right)^{\lambda-r}\p^{\lambda-r},\p^{\lambda-r}\right\rangle,\numberthis \label{39}
\end{align*}
\begin{align*}
M_2:=\sum_{r\in\z}\sum_{k,l\in\z}c_{k,l}\overline{c_{k,l+1}}e^{-2\pi i(\lambda-r)k}\left\langle\left(T_{(0,-1)}^t\right)^{\lambda-r}\p^{\lambda-r},\p^{\lambda-r}\right\rangle,
\end{align*}
and
\begin{align*}
M_3:=\sum_{k,l\in\z}|c_{k,l}|^2\sum_{r\in\z}\left\|\p^{\lambda-r}\right\|_{L^2(\R^2)}^2.
\end{align*}
We observe that $M_2=\overline{M_1}$. Hence, $\langle G(\lambda)\{c_{k,l}\},\{c_{k,l}\}\rangle_{\ell^2(\z^2)}=2\re{(M_1)}+M_3$. But $\re(M_1)\leq|M_1|$. Applying the \emph{Cauchy-Schwarz} inequality in \eqref{39}, we obtain $\re(M_1)\leq\left\|\{c_{k,l}\}\right\|_{\ell^2(\z^2)}^2I_{1,\lambda}$, where
\begin{align*}
I_{1,\lambda}&:=\left|\sum_{r\in\z}\left\langle\left(T_{(0,1)}^t\right)^{\lambda-r}\p^{\lambda-r},\p^{\lambda-r}\right\rangle\right|\\
&=\left|\sum_{r\in\z}|\widehat{h}(-(\lambda-r))|^2\int_0^2e^{\pi i(\lambda-r)x}\,dx\right|\\
&=\left|\sum_{r=1}^p\int_0^2e^{\pi i(\lambda-r)x}\,dx\right|.
\end{align*}
But,
\begin{align*}
\int_0^2e^{\pi i(\lambda-r)x}\,dx=2e^{\pi i(\lambda-r)}\sinc(\lambda-r).
\end{align*}
Hence,
\begin{align*}
I_{1,\lambda}\leq 2\sum_{r=1}^p|\sinc(\lambda-r)|\leq 2\sum_{r=1}^p 1=2p.
\end{align*}
As $\left\|\p^{\lambda-r}\right\|_{L^2(\R^2)}^2=|\widehat{h}(-(\lambda-r))|^2$,
\begin{align*}
M_3=2\left[\sum_{r\in\z}|\widehat{h}(-(\lambda-r))|^2\right]\left\|\{c_{k,l}\}\right\|_{\ell^2(\z^2)}^2=2p\left\|\{c_{k,l}\}\right\|_{\ell^2(\z^2)}^2.
\end{align*}
Therefore, 
\begin{align*}
\langle G(\lambda)\{c_{k,l}\},\{c_{k,l}\}\rangle_{\ell^2(\z^2)}&\leq 6p\left\|\{c_{k,l}\}\right\|_{\ell^2(\z^2)}^2.
\end{align*}
On the other hand, $\re(M_1)\geq -|M_1|$ leads to
\begin{align*}
&\langle G(\lambda)\{c_{k,l}\},\{c_{k,l}\}\rangle_{\ell^2(\z^2)}\geq 2\left[p-2\left|\sum_{r=1}^pe^{\pi i(\lambda-r)}\sinc(\lambda-r)\right|\right]\left\|\{c_{k,l}\}\right\|_{\ell^2(\z^2)}^2.
\end{align*}
Now,
\begin{align*}
& p-2\left|\sum_{r=1}^pe^{\pi i(\lambda-r)}\sinc(\lambda-r)\right|\\
&=p-2\left|\left(\sum_{r=1}^p\cos\left(\pi(\lambda-r)\right)\sinc(\lambda-r)\right)+i\left(\sum_{r=1}^p\sin\left(\pi(\lambda-r)\right)\sinc(\lambda-r)\right)\right|\\
&=:p-A_p(\lambda)\numberthis \label{40}.
\end{align*}
Employing some properties of the digamma function \cite[Section 6.3]{AS}
\[
\psi ^{(0)}(z) := \frac{d}{dz}\log\Gamma(z), \quad \re z > 0,
\]
we deduce that
\begin{align*}
\sum_{r=1}^p\cos\left(\pi(\lambda-r)\right)&\sinc(\lambda-r) = -\sum_{r=1}^p \frac{\cos(\pi\lambda)\sin(\pi\lambda)}{\pi(r-\lambda)}\\
& = -\left(\frac{\cos(\pi\lambda)\sin(\pi\lambda)}{\pi(1-\lambda)} + \frac{\sin (\pi  \lambda ) \cos (\pi  \lambda ) (\psi ^{(0)}(p-\lambda +1)-\psi
   ^{(0)}(2-\lambda ))}{\pi }\right)
\end{align*}
and
\begin{align*}
\sum_{r=1}^p\sin\left(\pi(\lambda-r)\right)&\sinc(\lambda-r) = \sum_{r=1}^p \frac{\sin^2(\pi\lambda)}{\pi(r-\lambda)}\\
& = \frac{\sin^2(\pi\lambda)}{\pi(1-\lambda)}+\frac{\sin ^2(\pi  \lambda ) (\psi ^{(0)}(p-\lambda +1)-\psi ^{(0)}(2-\lambda ))}{\pi }.
\end{align*}
Hence,
\begin{align*}
A_p(\lambda) &= 2\frac{\sin(\pi  \lambda )}{\pi (1-\lambda)} \big(1-(1-\lambda) \psi ^{(0)}(2-\lambda )+(1-\lambda) \psi^{(0)}(p-\lambda +1)\big)\\
& = 2 \sinc (1-\lambda) \big[1+(1-\lambda) \big(\psi^{(0)}(p-\lambda +1) - \psi ^{(0)}(2-\lambda) \big) \big],
\end{align*}
where we used that $\sin(\pi  \lambda ) = \sin\pi(1-  \lambda )$.

The goal is to find those values of $p$ for which $p - 2A_p(\lambda) > 0$, for all $\lambda\in (0,1]$. As the digamma function is monotone increasing and positive for integer arguments $\geq 2$ and as
\be\label{7.10}
\lim_{\lambda\to 0+} A_p(\lambda) = 0\quad\text{and}\quad A_p(1) = 2,
\ee
we show that $p - A_p(\lambda)$ has a unique positive minimum at $\lambda_0\in (0,1)$ whose value is strictly positive for $p\geq 3$ and that $p+1 - A_{p+1}(\lambda) > p - A_p(\lambda)$, for all $\lambda\in (0,1)$ and $p\geq 3$. To establish the latter, note that
\begin{align*}
p+1 - A_{p+1}(\lambda) & = p+1 - (2 \sinc (1-\lambda)) \big[1+(1-\lambda) \big(\psi^{(0)}(p+1-\lambda +1) - \psi ^{(0)}(2-\lambda \big) \big]\\
& = p+1 - (2 \sinc (1-\lambda))\big[1+(1-\lambda) \big(\psi^{(0)}(p+1-\lambda) + \frac{1}{p+1-\lambda} - \psi ^{(0)}(2-\lambda \big) \big]\\
&= p -A_p(\lambda) + 1 - \frac{2(1-\lambda)}{p+1-\lambda} \sinc (1-\lambda)\\
& > p -A_p(\lambda),\quad\text{for $p\geq 3$}.
\end{align*}
Hence, it suffices to show that $3 - A_3(\lambda)$ has a unique minimum value for a $\lambda\in [0,1]$. To this end, we remark that
\begin{align*}
3 - A_3(\lambda) &= 3 - (2 \sinc (1-\lambda)) \big[1+(1-\lambda) \big(\psi^{(0)}(4-\lambda +1) - \psi ^{(0)}(2-\lambda \big) \big]\\
& = 3 - (2 \sinc (1-\lambda)) \left[3 - \frac{2}{3-\lambda} - \frac{1}{2 - \lambda}\right] =: \Psi(\lambda).
\end{align*}
Differentiation of $\Psi$ with respect to $\lambda$ yields
\[
\Psi'(\lambda) = \frac{2 \pi  (\lambda -3) (\lambda -2) (\lambda -1) (3 (\lambda -4) \lambda +11) \cos
   (\pi  \lambda )-2 (3 (\lambda -4) \lambda  ((\lambda -4) \lambda +8)+49) \sin (\pi 
   \lambda )}{\pi  (\lambda -3)^2 (\lambda -2)^2 (\lambda -1)^2}.
\]
Numerically solving $\Psi'(\lambda) = 0$, $0 < \lambda < 1$, produces a unique zero at $\lambda_0 \approx 0.762714$. As $\Psi''(\lambda_0) \approx 12.8421$ and because of Eqns.~\ref{7.10}, the point $(\lambda_0, 3 - A_3(\lambda_0)) \approx (0.762714, 0.638135)$ is the unique global minimum of $3 - A_3(\lambda)$ on $[0,1]$.

%A simple computation shows that
%\begin{align*}
%A_p^2(\lambda)&=4\sum_{r=1}^p\sinc^2(\lambda-r)+8\sum_{r\neq r^\prime}\frac{\sin\left(\pi(\lambda-r)\right)\sin\left(\pi(\lambda-r^\prime)\right)}{\pi^2(\lambda-r)(\lambda-r^\prime)}\times\\
%&\left(\cos\left(\pi(\lambda-r)\right)\cos\left(\pi(\lambda-r^\prime)\right)+\sin\left(\pi(\lambda-r)\right)\sin\left(\pi(\lambda-r^\prime)\right)\right)\\
%&=4\left(\sum_{r=1}^p\sinc^2(\lambda-r)+2\sum_{r\neq r^\prime}(-1)^{r^\prime-r}\sinc\left(\pi(\lambda-r)\right)\sinc\left(\pi(\lambda-r^\prime)\right)\right)\\
%&=4\left(\sum_{r=1}^p(-1)^r\sinc(\lambda-r)\right)^2,
%\end{align*}
Therefore, the right-hand side of \eqref{40} is strictly positive. Hence, by Theorem \ref{d}, we conclude that the shift-invariant system $\{L_{(2k,l,m)}\p:k,l,m\in\z\}$ forms a \emph{Riesz} sequence for each $p\geq 3$. 
\end{example}
The following result shows that one can obtain more examples of \emph{Riesz} sequences of left translates on $\H$ from the \emph{Riesz} sequence of classical translates on $\R$.
\begin{proposition}\label{i}
Let $h\in L^2(\R)$. Define $\p(x,y,t) :=\chi_{[0,2]}(x)\chi_{[0,1]}(y)h(t)$. Then, the system $\{L_{(2k,l,m)}\p:k,l,m\in\z\}$ is a Riesz sequence in $L^2(\H)$ with bounds $A, B > 0$ iff the system $\{T_rh:r\in\z\}$ is a Riesz sequence in $L^2(\R)$ with bounds $\frac{1}{2}A$ and $\frac{1}{2}B$. 
\end{proposition}
\begin{proof}
We have $\p^\lambda(x,y)=\chi_{[0,2]}(x)\chi_{[0,1]}(y)\widehat{h}(-\lambda)$. Now,
\begin{align*}
\left\langle\left(T_{(2k,l)}^t\right)^\lambda\p^\lambda,\p^\lambda\right\rangle&=\int_{\R^2}e^{\pi i\lambda(lx-2ky)}\p^\lambda(x-2k,y-l)\overline{\p^\lambda(x,y)}\,dx\,dy\\
&=\overline{\widehat{h}(-\lambda)}\int_0^2\int_0^1e^{\pi i\lambda(lx-2ky)}\p^\lambda(x-2k,y-l)\,dy\,dx\\
&=\overline{\widehat{h}(-\lambda)}\int_{-2k}^{2-2k}\int_{-l}^{1-l}e^{\pi i\lambda(lx-2ky)}\p^\lambda(x,y)\,dy\,dx\\
&=|\widehat{h}(-\lambda)|^2\int_{[-2k,2-2k]\cap[0,2]}\int_{[-l,1-l]\cap[0,1]}e^{\pi i\lambda(lx-2ky)}\,dy\,dx,
\end{align*}
which in turn implies that $\langle(T_{(2k,l)}^t)^\lambda\p^\lambda,\p^\lambda\rangle=0$, $\forall\ (k,l)\in \z^{2}\setminus\{(0,0)\}$. Moreover, for $(k,l)=(0,0)$, $\langle(T_{(2k,l)}^t)^\lambda\p^\lambda,\p^\lambda\rangle=2|\widehat{h}(-\lambda)|^2$. For $\{c_{k,l}\}\in\ell^2(\z^{2})$, the middle term in \eqref{9} becomes
\begin{align*}
\langle G(\lambda)\{c_{k,l}\},\{c_{k,l}\}\rangle_{\ell^2(\z^2)}&=\sum_{k,l\in\z}|c_{k,l}|^2\sum_{r\in\z}2|\widehat{h}(-(\lambda-r))|^2\\
&=2\|\{c_{k,l}\}\|_{\ell^2(\z^2)}^2\sum_{r\in\z}|\widehat{h}(-(\lambda-r))|^2.
\end{align*}
From Theorem \ref{d}, the system $\{L_{(2k,l,m)}\p:k,l,m\in\z\}$ is a \emph{Riesz} sequence with bounds $A,B > 0$ iff
\begin{align*}
A\|\{c_{k,l}\}\|_{\ell^2(\z^2)}^2\leq 2\|\{c_{k,l}\}\|_{\ell^2(\z^2)}^2\sum_{r\in\z}|\widehat{h}(-(\lambda-r))|^2\leq B\|\{c_{k,l}\}\|_{\ell^2(\z^2)}^2,
\end{align*}
for a.e. $\lambda\in(0,1]$, which is equivalent to
\begin{align*}
\frac{A}{2}\leq \sum_{r\in\z}|\widehat{h}(-(\lambda-r))|^2\leq\frac{B}{2},
\end{align*}
for a.e. $\lambda\in(0,1]$. Hence, the required result follows from Theorem \ref{h}.
\end{proof}
\begin{example}\label{g}
Let $\p(x,y,t):=\chi_{[0,2]}(x)\chi_{[0,1]}(y)B_n(t)$, where $B_n$ denotes the cardinal polynomial $B$-spline of order $n$.
%Then by a straight forward computation one can show that
%\begin{align*}
%\langle G(\lambda)\{c_{k,l}\},\{c_{k,l}\}\rangle_{\ell^2(\z^2)}&=\sum_{k,l\in\z}|c_{k,l}|^2\sum_{r\in\z}2|\widehat{h}(-(\lambda-r))|^2\\
%&=2\left\|\{c_{k,l}\}\right\|_{\ell^2(\z^2)}^2\sum_{r\in\z}|\widehat{h}(-(\lambda-r))|^2.
%\end{align*}
It is well known that $\{T_rB_n:r\in\z\}$ is a \emph{Riesz} sequence in $L^2(\R)$, for each $n\in\N$. Hence, it follows from Proposition \ref{i} that $\{L_{(2k,l,m)}\p:k,l,m\in\z\}$ is a \emph{Riesz} sequence.
\end{example}
\section{Oblique dual of the system of left translates}
In this section, investigate the structure of an oblique dual of the system of left translates $\{L_{(2k,l,m)}\p:k,l,m\in\z\}$.
\begin{lemma}\label{b}
Assume that $\p,\widetilde{\p}\in\L$ have compact support and $\{L_{(2k,l,m)}\p:k,l,m\in\z\}$ and $\{L_{(2k,l,m)}\widetilde{\p}:k,l,m\in\z\}$ form {\it Riesz} sequences. Then, the following statements are equivalent:
\begin{enumerate}
\item[\emph{(i)}] $f=\sum\limits_{k,l,m\in\z}\langle f,L_{(2k,l,m)}\widetilde{\p}\rangle L_{(2k,l,m)}\p$, $\;\;\forall\ f\in V:=\overline{\Span\{L_{(2k,l,m)}\p:k,l,m\in\z\}}$.
\item[\emph{(ii)}] $\langle \p,L_{(2k,l,m)}\widetilde{\p}\rangle=\delta_{(k,l,m),(0,0,0)}$, $\;\;\forall\ (k,l,m)\in\z^3$.
\end{enumerate} 
\end{lemma}
\begin{proof}
The proof of this lemma is similar to the proof of Lemma $2.1$ in \cite{okl}. However, for the sake of completeness, we provide the proof. Suppose that (i) holds. As (i) is true for $f=\p$, we have
\begin{align*}
\p=\sum_{k,l,m\in\z}\langle \p,L_{(2k,l,m)}\widetilde{\p}\rangle L_{(2k,l,m)}\p,
\end{align*}
which leads to
\begin{align*}
[\langle \p,L_{(2k,l,m)}\widetilde{\p}\rangle-1]\p\ +\sum_{\substack{k,l,m\in\z \\ (k,l,m)\neq(0,0,0)}}\langle \p,L_{(2k,l,m)}\widetilde{\p}\rangle L_{(2k,l,m)}\p=0.
\end{align*}
As $\{L_{(2k,l,m)}\p:k,l,m\in\z\}$ is a {\it Riesz} sequence, we know that $\langle \p,L_{(2k,l,m)}\widetilde{\p}\rangle=\delta_{(k,l,m),(0,0,0)}$, $\forall\ (k,l,m)\in\z^3$, which is (ii).

Conversely, suppose (ii) holds. Let $f\in V$. Then $f=\sum\limits_{k,l,m\in\z}c_{k,l,m}L_{(2k,l,m)}\p$, for some coefficients $\{c_{k,l,m}\}$. Now,
\begin{align*}
\langle f,L_{(2k,l,m)}\widetilde{\p}\rangle&=\sum_{k^\prime,l^\prime,m^\prime\in\z}c_{k^\prime,l^\prime,m^\prime}\langle L_{(2k^\prime,l^\prime,m^\prime)}\p,L_{(2k,l,m)}\widetilde{\p}\rangle\\
&=\sum_{k^\prime,l^\prime,m^\prime\in\z}c_{k^\prime,l^\prime,m^\prime}\langle \p,L_{(2(k-k^\prime),l-l^\prime,m-m^\prime+(k^\prime l-l^\prime k))}\widetilde{\p}\rangle\\
&=\sum_{k^\prime,l^\prime,m^\prime\in\z}c_{k^\prime,l^\prime,m^\prime}\delta_{(k-k^\prime,l-l^\prime,m-m^\prime+(k^\prime l-l^\prime k)),(0,0,0)}\\
&=c_{k,l,m},
\end{align*}
from which (i) follows.
\end{proof}
\begin{theorem}\label{e}
Let $\p\in\L$ be supported in $[0,2n]\times[0,n]\times[0,M]$, for some $M,n\in\N$. Also assume that the system $\{L_{(2k,l,m)}\p:k,l,m\in\z\}$ forms a \emph{Riesz} sequence. Then, the following statements are equivalent:
\begin{enumerate}
\item[\emph{(i)}] The system $\{L_{(2k,l,m)}\p:k,l,m\in\z\}$ has a generalized dual $\{L_{(2k,l,m)}\widetilde{\p}:k,l,m\in\z\}$ with $\supp\widetilde{\p}\subset Q$, where $Q:=[0,2]\times [0,1]\times [0,1]$.
\item[\emph{(ii)}] If $\sum\limits_{(k,l,m)\in A}c_{k,l,m}L_{(2k,l,m)}\p(x,y,t)=0$, for all $(x,y,t)\in Q$ and for some coefficients $\{c_{k,l,m}\}$, then $c_{0,0,0}=0$, where $A:=\{-(n-1)\leq k,l\leq 0, -M-n+1<m<n\}$.
\item[\emph{(iii)}] $\p|_Q\notin \Span\left\{(L_{(2k,l,m)}\p)|_Q:(k,l,m)\in A\setminus\{(0,0,0)\}\right\}$. 
\end{enumerate}
In case that any one of the above conditions is satisfied, the generalized duals $\{L_{(2k,l,m)}\widetilde{\p}:k,l,m\in\z\}$ form orthogonal sequences and they are oblique duals of $\{L_{(2k,l,m)}\p:k,l,m\in\z\}$. One can choose $\widetilde{\p}$ to be of the form 
\begin{align*}
\widetilde{\p}=\bigg[\sum_{(k,l,m)\in A}d_{k,l,m}L_{(2k,l,m)}\p\bigg]\chi_Q,
\end{align*}
for some coefficients $\{d_{k,l,m}\}$. Here, $\chi_Q$ denotes the characteristic function of $Q$.
\end{theorem}
\begin{proof}
The idea of the proof is similar to that of Theorem $3.1$ of \cite{okl}. Here, we provide the main steps in the proof. 

Let $\widetilde{\p}\in\L$ be such that $\supp\widetilde{\p}\subset Q$. Then,
\begin{align*}
\langle L_{(2k,l,m)}\p,\widetilde{\p}\rangle&=\int_QL_{(2k,l,m)}\p(x,y,t)\overline{\widetilde{\p}}(x,y,t)~dxdydt\\
&=\int_{-2k}^{2(1-k)}\int_{-l}^{1-l}\int_{-m+\frac{1}{2}(2ky-lx)}^{1-m+\frac{1}{2}(-lx+2ky)}\p(x,y,t)\overline{\widetilde{\p}}(x+2k,y+l,\\
&\hspace{8 cm} t+m-\frac{1}{2}(-lx+2ky))~dtdydx,
\end{align*}
by applying a change of variables. Further, using $\supp\p\subset [0,2n]\times[0,n]\times[0,M]$, we obtain $\langle L_{(2k,l,m)}\p,\widetilde{\p}\rangle=0,\ \forall\ (k,l,m)\in A^c$. 

Assume that (i) holds. Then, by Lemma \ref{b}, we have that $\langle \p,L_{(2k,l,m)}\widetilde{\p}\rangle=\delta_{(k,l,m),(0,0,0)}$, $\forall\ (k,l,m)\in\z^3$. Hence, we obtain the moment problem
\begin{align*}
\langle L_{(2k,l,m)}\p,\widetilde{\p}\rangle=\delta_{(k,l,m),(0,0,0)},
\end{align*}
for $(k,l,m)\in A$. Now, condition (i) is equivalent to the existence of a solution of the moment problem. By Lemma \ref{c}, the existence of a solution of the moment problem is equivalent to conditions (ii) and (iii). Moreover, if (i) is true, then $\supp\widetilde{\p}\subset Q$ leads to the fact that the system $\{L_{(2k,l,m)}\widetilde{\p}:k,l,m\in\z\}$ is an orthogonal sequence.
\end{proof} 
\begin{example}
Let $\p(x,y,t):=\chi_{[0,2]}(x)\chi_{[0,1]}(y)B_3(t)$, where $B_3$ is the cardinal polynomial $B$-spline of order $3$, given by
\begin{align*}
B_3(t)=\begin{cases}
\frac{1}{2}t^2, & t\in[0,1];\\
-t^2+3t-\frac{3}{2}, & t\in[1,2];\\
\frac{1}{2}t^2-3t+\frac{9}{2}, & t\in[2,3];\\
0, & \text{otherwise}.
\end{cases}
\end{align*}
Thus, it follows from Example \ref{g} that $\{L_{(2k,l,m)}\p:k,l,m\in\z\}$ is a \emph{Riesz} sequence. We know that $\supp\widetilde{\p}\subset Q$. Consider
\begin{align*}
\langle L_{(2k,l,m)}\p,\widetilde{\p}\rangle&=\int_0^2\int_0^1\int_0^1\p(x-2k,y-l,t-m+\tfrac{1}{2}(2ky-lx))\overline{\widetilde{\p}(x,y,t)}\,dt\,dy\,dx\\
&=\int_{[-2k,2-2k]\cap[0,2]}\int_{[-l,1-l]\cap[0,1]}\int_0^1B_3(t-m+\tfrac{1}{2}(2ky-lx))\overline{\widetilde{\p}(x+2k,y+l,t)}\,dt\,dy\,dx.
\end{align*}
Hence, for $(k,l)\neq(0,0)$, $\langle L_{(2k,l,m)}\p,\widetilde{\p}\rangle=0$. 

For $(k,l)=(0,0)$, we have
\begin{align*}
\langle L_{(0,0,m)}\p,\widetilde{\p}\rangle=\int_0^2\int_0^1\int_{[-m,1-m]\cap[0,3]}B_3(t)\overline{\widetilde{\p}(x.y,t+m)}\,dt\,dy\,dx,
\end{align*}
which shows that $-2\leq m\leq 0$. Define $\Lambda:=\{(0,0,-2),(0,0,-1),(0,0,0)\}$. Then, $\langle L_{(0,0,m)}\p,\widetilde{\p}\rangle=0$, $\forall\ (k,l,m)\notin\Lambda$. Furthermore, it is easy to show that $\left\{\p|_Q, \left(L_{(0,0,-1)}\p\right)|_Q,\left(L_{(0,0,-2)}\p\right)|_Q\right\}$ is a linearly independent set. Thus, by Theorem \ref{e}, an oblique dual of $\p$ is given by
\begin{align*}
\widetilde{\p}=\bigg[\sum_{m=-2,-1,0}d_{m}L_{(0,0,m)}\p\bigg]\chi_Q,\numberthis \label{42}
\end{align*}
satisfying the moment problem
\begin{align*}
\langle L_{(0,0,m)}\p,\widetilde{\p}\rangle=\delta_{0,m}.\numberthis \label{41}
\end{align*}
for $m=-2,-1,0$. 

Next, we proceed to solve the above moment problem. Substituting \eqref{42} in \eqref{41}, we get the following equations
\begin{align*}
\sum_{m=-2,-1,0}\overline{d_m}\ \langle \p,L_{(0,0,m)}\p\cdot\chi_Q\rangle=1,
\end{align*}
\begin{align*}
\sum_{m=-2,-1,0}\overline{d_m}\ \langle L_{(0,0,-1)}\p,L_{(0,0,m)}\p\cdot\chi_Q\rangle=0,
\end{align*}
\begin{align*}
\sum_{m=-2,-1,0}\overline{d_m}\ \langle L_{(0,0,-2)}\p,L_{(0,0,m)}\p\cdot\chi_Q\rangle=0.
\end{align*}
Upon simplification, we obtain
\begin{align*}
&6d_0+13d_{-1}+d_{-2}=60,\\
&d_0+\frac{54}{13}d_{-1}+d_{-2}=0,\\
&d_0+13d_{-1}+6d_{-2}=0.
\end{align*}
Solving these equations and then substituting back into \eqref{42}, yields
\begin{align*}
\widetilde{\p}(x,y,t)=\tfrac{3}{2}(40t^2-36t+5)\chi_Q(x,y,t).
\end{align*}
\end{example}
\bibliographystyle{amsplain}
\bibliography{Hsplines-2}
\end{document}